\documentclass[10pt]{article}
\usepackage{CJK}
\usepackage{amsfonts}
\usepackage{amssymb}
\usepackage{eucal}
\usepackage{amsthm}
\usepackage{amsmath}
\usepackage{authblk}
\usepackage{enumerate}
\usepackage{graphicx}
\usepackage{subfigure}
\usepackage{dsfont}
\usepackage{caption}
\usepackage[all]{xy}
\usepackage{bm}
\usepackage{float}
\usepackage{overpic}
\usepackage{esvect}
\usepackage{epstopdf}
\usepackage{color}
\usepackage{geometry}
\usepackage{listings,framed}
\usepackage{indentfirst}
\usepackage{lipsum}
\usepackage{geometry}
\usepackage[utf8]{inputenc}
\usepackage[colorlinks,
            linkcolor=blue,
            anchorcolor=blue,
            citecolor=blue]{hyperref}
\theoremstyle{plain}
\newtheorem{Definition}{Definition}[section]
\newtheorem{Theorem}{Theorem}[section]%标序
\newtheorem{Lemma}{Lemma}[section]%标序
\newtheorem{Example}{Example}[section]%标序
\newtheorem{Remark}{Remark}[section]

\begin{document}
\title{\normalsize \textbf {Equivariant Mapping Class Group and Orbit Braid Group}\footnote{Partially supported by the NSFC grants(No. 11971112).}}

\author{\normalsize \textmd{Shuya Cai}\footnote{Corresponding author} and \textmd{Hao Li}}

\affil{\setlength{\baselineskip}{20pt} \emph{Department of Mathematics, Fudan University} \\ \emph{19110180014@fudan.edu.cn} \\ \emph{14110840001@fudan.edu.cn}}

\date{}

\maketitle

\renewcommand{\abstractname}

\begin{abstract}
\footnotesize \setlength{\baselineskip}{10pt}
\setlength{\oddsidemargin}{14 in}
\noindent
Motivated by the work of Birman about the relationship between mapping class groups and braid groups, we discuss the relationship between the orbit braid group and the equivariant mapping class group on the closed surface $M$ with a free and proper group action in this article. Our construction is based on the exact sequence given by the fibration $\mathcal{F}_{0}^{G}M \rightarrow F(M/G,n)$. The conclusion is closely connected with the braid group of the quotient space. Comparing with the situation without the group action, there is a big difference when the quotient space is $\mathbb{T}^{2}$.
\end{abstract}

{\footnotesize
\setlength{\oddsidemargin}{14 in}
\setlength{\baselineskip}{10pt}
\setlength{\parindent}{2em}
  \emph{keywords}: Equivariant mapping class group; Orbit braid group; Evaluation map; Center of braid group

  }

\section{Introduction}

Braid groups of the plane were defined by Artin in 1925~\cite{ref1}, and further studied in ~\cite{ref2},~\cite{ref3}. Braid groups of surfaces were studied by Zariski~\cite{ref19}, and were later generalized using the following definition due to Fox~\cite{ref11}. Bellingeri simplified presentations of braid groups and pure braid groups on surfaces and showed some propertities of surface pure braid groups in ~\cite{ref4}.

Let $M$ be the closed surface, with a finite set $\mathcal{P}$ of $n$ distinguished points in $M$. $\mathcal{B}_{n}M$
\noindent(respectively $\mathcal{F}_{n}M$) denotes all homeomorphsims $M\rightarrow M$ which preserve the set $\mathcal{P}$ (respectively, each point in $\mathcal{P}$) and preserve the orientation if $M$ is oriented. $\pi_{0}(\mathcal{F}_{n}M)$ is the set of all path connected components of $\mathcal{F}_{n}M$ and the $n$-th mapping class group denoted by $Mod(M,n)$ is the set of all path connected components of $\mathcal{B}_{n}M$. The algebraic structure of the mapping class group is of great importance in the theory of Riemann surfaces. Connections between the mapping class group and the braid group of closed surfaces have been discovered, in order to help find the generators and relations of their mapping class groups.

The (pure) braid group of $M$ on $n$ strands is denoted by $B_{n}M$ ($P_{n}M$). Let $S_{g}$ be the oriented closed surface of genus $g$, with $g\geq0$. In 1969, Birman~\cite{ref5}~\cite{ref6} gave the basic relationship between mapping class groups and braid groups on $S_{g}$.
\begin{Theorem}[Birman,1969]
Let $i_{\ast}:\pi_{0}(\mathcal{F}_{n}S_{g})\rightarrow \pi_{0}(\mathcal{F}_{0}S_{g})$ be the homomorphism induced by inclusion $\mathcal{F}_{n}S_{g}\subset \mathcal{F}_{0}S_{g}$. Then
\begin{align*}
&\ker i_{\ast}\cong P_{n}S_{g} \quad if \quad g\geq2;\\
&\ker i_{\ast}\cong P_{n}S_{g}/Z(P_{n}S_{g}) \quad if \quad g=1,n\geq2 \quad or \quad g=0,n\geq3.
\end{align*}
\end{Theorem}
\begin{Theorem}[Birman,1969]
Let $i_{\ast}:Mod(S_{g},n)\rightarrow Mod(S_{g},0)$ be the homomorphism induced by inclusion $\mathcal{B}_{n}S_{g}\subset \mathcal{B}_{0}S_{g}$. Then
\begin{align*}
&\ker i_{\ast}\cong B_{n}S_{g} \quad if \quad g\geq2;\\
&\ker i_{\ast}\cong B_{n}S_{g}/Z(B_{n}S_{g}) \quad if \quad g=1,n\geq2 \quad or \quad g=0,n\geq3.
\end{align*}
\end{Theorem}
Using above results, Birman obtained a full set of generators for the mapping class groups of $n$-punctured oriented 2-manifolds. Then she computed the mapping class group of the $n$-punctured sphere and gave relations for the mapping class group of torus in a new way.

Denote the nonorientable closed surface of genus $k$ as $N_{k}$ with $k\geq1$. $S_{k-1}$ is its orientable double covering and let $\pi:S_{k-1}\rightarrow N_{k}$ be the covering map. The induced homomorphsim between braid groups $\varphi_{n}:B_{n}(N_{k})\rightarrow B_{2n}(S_{k-1})$ is injective~\cite{ref13} on the level of fundamental group. The homomorphism between mapping class groups $\phi_{n}:Mod(N_{k},n)\rightarrow Mod(S_{k-1},2n)$ induced by $\pi$ is also injective. Furthermore, if $k\geq3$ then we have a commutative diagram of the following form:\\
\xymatrix{&1 \ar[r] &B_{n}(N_{k}) \ar[d]^{\varphi_{n}} \ar[r] &Mod(N_{k},n) \ar[d]^{\phi_{n}} \ar[r]^{\psi_{n}} &Mod(N_{k},0) \ar[d]^{\phi_{0}} \ar[r] &1 \\
&1 \ar[r] &B_{2n}(S_{k-1}) \ar[r] &Mod(S_{k-1},2n) \ar[r]^{\widetilde{\psi}_{n}} &Mod(S_{k-1},0) \ar[r] &1}
\\where $\psi_{n}$ and $\widetilde{\psi}_{n}$ are the homomorphsims induced by inclusions $\mathcal{B}_{n}N_{k}\subset \mathcal{B}_{0}N_{k}$ and $\mathcal{B}_{2n}S_{k-1}\subset \mathcal{B}_{0}S_{k-1}$~\cite{ref14}.

In this article, we study the relationship between the orbit braid group and the equivariant mapping class group on the closed surface $M$ which admits a group action. Let $G$ be a discrete group and act on $M$ freely and properly. Consider a finite set $\mathcal{P}=\{\mathbf{x_{1}},\cdots,\mathbf{x_{n}}\}$ of $n$ arbitrarily chosen points on $M$ with different orbits. Define $\mathcal{B}_{n}^{G}M$ (respectively, $\mathcal{F}_{n}^{G}M$) to be the group of all $G$-homemorphisms $f:M\rightarrow M$ which satisfy $f(G\mathcal{P})=G\mathcal{P}$ (respectively, $f(G\mathbf{x_{i}})=G\mathbf{x_{i}},i=1,\cdots,n$) and preserve the orientation if $M$ is oriented. These two groups are endowed with compact-open topology. $\pi_{0}(\mathcal{F}_{n}^{G}M,id)$ is the set of all path connected components of $\mathcal{F}_{n}^{G}M$. The equivariant mapping class group denoted by $M^{G}(M,n)$ is defined to be the set of all path connected components $\mathcal{B}_{n}^{G}M$.

The concept of the orbit braid group on an arbitrary connected topological manifold $M$ is discussed in ~\cite{ref16}. Let $M$ be a connected topological manifold of dimension at least 2 with an effective action of a finite group $G$. The orbit configuration space of $n$ ordered points in the $G$-space $M$ is defined by
$$F_{G}(M,n)=\{(\mathbf{x_{1}},\cdots,\mathbf{x_{n}})\in M^{n}:G(\mathbf{x_{i}})\bigcap G(\mathbf{x_{j}})=\emptyset \quad if \quad i\neq j\}$$
with subspace topology and $G(\mathbf{x})$ denotes the orbit of $x$. The action of $G$ on $M$ induces a natural action of $G^{n}$ on $F_{G}(M,n)$. There is also a canonical free action of the symmetric group $\Sigma_{n}$ on $F_{G}(M,n)$. Generally, these two actions are not commutative. Let $\pi_{1}^{E}(F_{G}(M,n),x,x^{orb})$ be the set consisting of the homotopy classes relative to $\partial I$ of all paths $\alpha:I\rightarrow F_{G}(M,n)$ with $\alpha(0)=x$ and $\alpha(1)\in x^{orb}$, where $x^{orb}=\{gx^{orb}:g\in G^{n},\sigma\in\Sigma_{n}\}$, which is the orbit set at $x$ under two actions of $G^{n}$ and $\Sigma_{n}$. The orbit braid group $B_{n}^{orb}(M,G)$ bijectively corresponds to $\pi_{1}^{E}(F_{G}(M,n),x,x^{orb})$. And the pure orbit braid group $P_{n}^{orb}(M,G)$ bijectively corresponds to $\pi_{1}^{E}(F_{G}(M,n),x,G^{n}(x))$. If the action $G$ on $M$ is free, $P_{n}^{orb}(M,G)\cong P_{n}(M/G)$ and $B_{n}^{orb}(M,G)\cong B_{n}(M/G)$. Thus the orbit braid group $B_{n}^{orb}(M,G)$ (respectively, $P_{n}^{orb}(M,G)$) we will discuss later corresponds to $B_{n}(M/G)$ (respectively, $P_{n}(M/G)$), where $M/G$ is an oriented or nonorientable closed surface~\cite{ref15}.

We prove the following results in this article:
\begin{Theorem}[Theorem 3.1-3.3]
Let $i_{\ast}^{G}:\pi_{0}(\mathcal{F}_{n}^{G}M)\rightarrow\pi_{0}(\mathcal{F}_{0}^{G}M)$ be the homomorphism induced by inclusion $\mathcal{F}_{n}^{G}M\subset \mathcal{F}_{0}^{G}M$. Then
\begin{align*}
\ker i_{\ast}^{G}&\cong P_{n}(M/G),\quad if\quad M/G \quad is \quad S_{g},\quad g\geq2 \quad or \quad N_{k},\quad k\geq2;\\
\ker i_{\ast}^{G}&\cong P_{n}(M/G)/Z(P_{n}(M/G)),\quad if\quad M/G \quad is \quad \mathbb{S}^{2} \quad or \quad \mathbb{R}P^{2};\\
\end{align*}
When $M/G$ is $\mathbb{T}^{2}$,
\begin{align*}
M&=\mathbb{T}^{2},\quad G=\mathds{Z}/q\mathds{Z}\bigoplus\mathds{Z}/r\mathds{Z},\quad q,r\geq1 \quad and\\
\ker i_{\ast}^{G}&\cong P_{n}(M/G)/<\widetilde{{a}}^{q},\widetilde{b}^{r}:\widetilde{{a}}^{q}\widetilde{b}^{r}=\widetilde{b}^{r}\widetilde{{a}}^{q}>. \end{align*}
\end{Theorem}
\begin{Theorem}[Theorem 3.4]
Let $j_{\ast}^{G}:Mod_{G}(M,n)\rightarrow Mod_{G}(M,0)$ be the homomorphism induced by inclusion $\mathcal{B}_{n}^{G}M\subset \mathcal{B}_{0}^{G}M$. Then
\begin{align*}
\ker j_{\ast}^{G}&\cong B_{n}(M/G),\quad if\quad M/G \quad is \quad S_{g},\quad g\geq2 \quad or \quad N_{k},\quad k\geq2;\\
\ker j_{\ast}^{G}&\cong B_{n}(M/G)/Z(P_{n}(M/G)),\quad if\quad M/G \quad is \quad \mathbb{S}^{2} \quad or \quad \mathbb{R}P^{2};\\
\end{align*}
When $M/G$ is $\mathbb{T}^{2}$,
\begin{align*}
M&=\mathbb{T}^{2},\quad G=\mathds{Z}/q\mathds{Z}\bigoplus\mathds{Z}/r\mathds{Z},\quad q,r\geq1 \quad and\\
\ker j_{\ast}^{G}&\cong B_{n}(M/G)/<\widetilde{{a}}^{q},\widetilde{b}^{r}:\widetilde{{a}}^{q}\widetilde{b}^{r}=\widetilde{b}^{r}\widetilde{{a}}^{q}>. \end{align*}
\end{Theorem}
This article is organized as follows. In Section 2 we conclude the centers of pure braid groups on oriented and nonorientable closed surfaces. Section 3 is the main part of this article. We establish the relationship between the orbit braid group and the equivariant mapping class group on the closed surface which admits a free and proper group action. The proofs of Theorem 1.3 and 1.4 are given. The conclusion is closely connected with the quotient space. And comparing with the situation without the group action, there is a big difference when the quotient space is $\mathbb{T}^{2}$.
\section{The centers of pure braid groups on closed surfaces}
First, we will briefly recall the definition and properties of braid groups. Let $M$ be a smooth manifold. The configuration space of $n$ ordered points in $M$, denoted $F(M,n)$ is defined as:
$$F(M,n):=\{(\mathbf{x_{1}},\cdots,\mathbf{x_{n}})\in M^{n}:\mathbf{x_{i}}\neq \mathbf{x_{j}}\quad if \quad i\neq j\}.$$
There is a natural action of the symmetric group $\Sigma_{n}$ on the space $F(M,n)$, given by permuting the coordinates. The configuration space of $n$ unordered points in $M$ is the quotient space:
$$C(M,n):=F(M,n)/\Sigma_{n}.$$
Following Fox and Neuwirth, the $n$th pure braid group $P_{n}(M)$ (respectively, the $n$th braid group $B_{n}(M)$) is defined to be the fundamental group of $F(M,n)$ (respectively, of $C(M,n)$)~\cite{ref11}.

If $m>n$ are positive integers, we can define a homomorphism $\theta_{\ast}:P_{m}(M)\rightarrow P_{n}(M)$ induced by the projective $\theta:F(M,m)\rightarrow F(M,n)$ defined by
$$\theta((\mathbf{x_{1}},\cdots,\mathbf{x_{m}}))=(\mathbf{x_{1}},\cdots,\mathbf{x_{n}}).$$
In \cite{ref9}, Fadell and Neuwirth study the map $\theta$, and show that it is a locally trivial fibration. The fiber over a point $(\mathbf{x_{1}},\cdots,\mathbf{x_{n}})$ of the base space is $F(M-\{\mathbf{x_{1}},\cdots,\mathbf{x_{n}}\},m-n)$. Applying the associtaed long exact homotopy sequence, we obtain the pure braid group exact sequence of Fadell and Neuwirth:
$$\cdots \pi_{2}(F(M,n))\rightarrow P_{m-n}(M-\{\mathbf{x_{1}},\cdots,\mathbf{x_{n}}\})\xrightarrow {\delta_{\ast}} P_{m}(M)\xrightarrow{\theta_{\ast}} P_{n}(M)\rightarrow 1.$$
The following short exact sequence is proved to be true where $n\geq3$ if $M=\mathbb{S}^{2}$, $n\geq2$ if $M=\mathbb{R}P^{2}$ and $n \geq1$ otherwise in $\cite{ref14}\cite{ref15}\cite{ref16}$:
$$1\rightarrow P_{m-n}(M-\{\mathbf{x_{1}},\cdots,\mathbf{x_{n}}\})\xrightarrow {\delta_{\ast}} P_{m}(M)\xrightarrow{\theta_{\ast}} P_{n}(M)\rightarrow 1.$$

Let $S_{g}$ be the oriented closed surface of genus $g$. Birman computed all the centers of pure braid groups on oriented closed surface in \cite{ref5} and \cite{ref6}:
\begin{Theorem}[Birman,1969]
\begin{align*}
&Z(P_{n}(S_{g}))=1,\quad g\geq2;\\
&Z(P_{n}(\mathbb{T}^{2}))=<\widetilde{a},\widetilde{b}|\widetilde{a}\widetilde{b}=\widetilde{b}\widetilde{a}>;\\
&Z(P_{n}(\mathbb{S}^{2}))=\mathds{Z}_{2}^{2},\quad n\geqslant3;\\
&Z(P_{1}(\mathbb{S}^{2}))=Z(P_{2}(\mathbb{S}^{2}))=1.
\end{align*}
\end{Theorem}
Let $N_{k}$ be the nonorientable closed surface of genus $k$. Next we will compute the center of $P_{n}(N_{k})$.
\begin{Theorem}
If $k\geq2$, then $Z(P_{n}(N_{k}))=1$.
\end{Theorem}
\begin{proof}
We apply an induction on the number $n$ of strands. For $n=1$,
$$P_{1}(N_{k})=\pi_{1}(N_{k})=<\rho_{1},\cdots,\rho_{k}|\prod\limits_{j=1}^{k}\rho_{j}^{2}=1>.$$
which is a finitely generated group with a single defining relation. Thus we obtain that $Z(P_{1}(N_{k}))=1$, if $k\geq2$~\cite{ref17}.
The Fadell-Neuwirth fibration gives us the exact sequence
$$1\rightarrow\pi_{1}(N_{k}-\{\mathbf{x_{1}},\cdots,\mathbf{x_{n}}\})\xrightarrow{\delta_{\ast}} P_{n+1}(N_{k})\xrightarrow{\theta_{\ast}} P_{n}(N_{k})\rightarrow 1$$
where $\pi_{1}(N_{k}-\{\mathbf{x_{1}},\cdots,\mathbf{x_{n}}\})$ is a free group, for $n\geq 1$. Suppose $Z(P_{n}(N_{k}))=1$. Since $p_{\ast}$ is surjective, $\theta_{\ast}(Z(P_{n+1}(N_{k})))\subset Z(P_{n}(N_{k}))=1$. Hence $Z(P_{n+1}(N_{k}))$ lies in the group of $\ker \theta_{\ast}=im \delta_{\ast}$, which is a free group. Thus $Z(P_{n+1}(N_{k}))=1$.
\end{proof}

In order to compute the center of $P_{n}(\mathbb{R}P^{2})$, we need to know the specific presentation of $P_{n}(\mathbb{R}P^{2})$.
\begin{Theorem}[D. L. Goncalves and J. Guaschi~\cite{ref12}]
The group $P_{n}(\mathbb{R}P^{2})$ admits the following presentation:
\begin{itemize}
  \item Generators: $B_{ij},\quad 1\leqslant i<j\leqslant n;\quad \rho_{k},\quad 1\leqslant k \leqslant n$.
  \item Relations:
  \begin{enumerate}[(a)]
  \item
  $B_{rs}B_{ij}B_{rs}^{-1}=\left\{
                             \begin{array}{ll}
                               B_{ij}, & {i<r<s<j} \\
                               B_{ij}^{-1}B_{rj}^{-1}B_{ij}B_{rj}B_{ij}, & {r<i=s<j} \\
                               B_{sj}^{-1}B_{ij}B_{sj}, & {i=r<s<j} \\
                               B_{sj}^{-1}B_{rj}^{-1}B_{sj}B_{rj}B_{ij}B_{rj}^{-1}B_{sj}^{-1}B_{rj}B_{sj}, & {r<i<s<j}
                             \end{array}
                           \right.$
  \item $\rho_{i}\rho_{j}\rho_{i}^{-1}=\rho_{j}^{-1}B_{ij}^{-1}\rho_{j}^{2},\quad 1\leqslant i<j\leqslant n$;
  \item $\rho_{i}^{2}=B_{1i}\cdots B_{i-1,i}B_{i,i+1}\cdots B_{in},\quad 1\leqslant i \leqslant n$;
  \item
  For $1\leqslant i<j\leqslant n,\quad 1\leqslant k \leqslant n,\quad k\neq j$,
  \\$\rho_{k}B_{ij}\rho_{k}^{-1}\left\{
                                      \begin{array}{ll}
                                        B_{ij}, & {j<k \quad or \quad k<i;} \\
                                        \rho_{j}^{-1}B_{ij}^{-1}\rho_{j}, & {k=i;} \\
                                        \rho_{j}^{-1}B_{kj}^{-1}\rho_{j}B_{kj}^{-1}B_{ij}B_{kj}\rho_{j}^{-1}B_{kj}\rho_{j}, & {i<k<j.}
                                      \end{array}
                                    \right.$
   \end{enumerate}
\end{itemize}
\end{Theorem}
We view the projective plane as the quotient space of the sphere. Then the figures of $B_{ij}$ and $\rho_{i}$ are given in Figure 1, 2.
\begin{figure}[htbp]
\subfigure{
\begin{minipage}[t]{0.5\linewidth}
\centering
\includegraphics[width=1\textwidth]{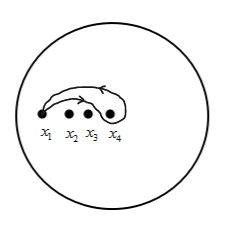}
\caption*{Figure 1}
\end{minipage}}
\subfigure{
\begin{minipage}[t]{0.5\linewidth}
\centering
\includegraphics[width=1\textwidth]{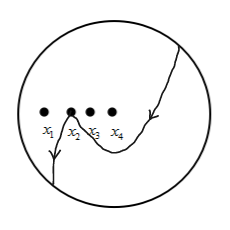}
\caption*{Figure 2}
\end{minipage}}
\end{figure}%几何解释，画图

From the Relation (b) and (d), each generator $B_{ij}$ in $P_{n}(\mathbb{R}P^{2})$ can be presented by $\rho_{i}$ and $\rho_{j}$:
$$B_{ij}=\rho_{j}\rho_{i}^{-1}\rho_{j}^{-1}\rho_{i},\quad 1\leqslant i<j\leqslant n.$$
With this presentation and Relation (b), we obtain
$$\rho_{i}\rho_{j}\rho_{i}\rho_{j}=\rho_{j}\rho_{i}\rho_{j}\rho_{i},\quad 1\leqslant i,j\leqslant n.$$

When $n=1$, $P_{1}(\mathbb{R}P^{2})=\pi_{1}(\mathbb{R}P^{2})=\mathbb{Z}_{2}$, which is an abelian group. Next we consider $n\geq2$.
\begin{Lemma}
For $n\geq2$, $\tau_{n}=\tau_{n1}\cdots\tau_{nn}$ lies in the center of $P_{n}(\mathbb{R}P^{2})$, where
$$\tau_{ni}=B_{i,i+1}B_{i,i+2}\cdots B_{in}=B_{i-1,i}^{-1}\cdots B_{1i}^{-1}\rho_{i}^{2},\quad i=1,\cdots,n.$$
\end{Lemma}
\begin{proof}
According to relation (d), $\tau_{ni}\rho_{k}=\rho_{k}\tau_{ni}$ for $k<i$. For any $i\neq k$, we have
\begin{align*}
B_{ik}^{-1}\rho_{k}B_{ik}&=(\rho_{i}^{-1}\rho_{k}\rho_{i}\rho_{k}^{-1})\rho_{k}(\rho_{k}\rho_{i}^{-1}\rho_{k}^{-1}\rho_{i})\\
&=\rho_{i}^{-1}(\rho_{k}\rho_{i}\rho_{k}\rho_{i}^{-1})\rho_{k}^{-1}\rho_{i}\\
&=\rho_{i}^{-1}(\rho_{i}^{-1}\rho_{k}\rho_{i})\rho_{k}\rho_{k}^{-1}\rho_{i}\\
&=\rho_{i}^{-2}\rho_{k}\rho_{i}^{2}.
\end{align*}
Thus we obtain
\begin{align*}
\tau_{nk}\rho_{k}\tau_{nk}^{-1}&=B_{k-1,k}^{-1}\cdots B_{1k}^{-1}\rho_{k}\cdots B_{1k}\cdots B_{k-1,k}\\
&=\rho_{1}^{-2}\cdots\rho_{k-1}^{-2}\rho_{k}\rho_{k-1}^{2}\cdots\rho_{1}^{2}.
\end{align*}
From the above results, we obtain
\begin{align*}
\tau_{n}\rho_{k}&=\tau_{n1}\cdots\tau_{nn}\rho_{k}\\
&=\tau_{n1}\cdots\tau_{n,k-1}\rho_{1}^{-2}\cdots\rho_{k-1}^{-2}\rho_{k}\rho_{k-1}^{2}\cdots\rho_{1}^{2}\tau_{nk}\cdots\tau_{nn}\\
&=(\tau_{n2}\rho_{2}^{-2})\cdots(\tau_{n,k-1}\rho_{k-1}^{-2})\rho_{k}\rho_{k-1}^{2}\cdots\rho_{1}^{2}\tau_{nk}\cdots\tau_{nn}\\
&=B_{12}^{-1}(B_{23}^{-1}B_{13}^{-1})\cdots(B_{k-2,k-1}^{-1}\cdots B_{1,k-1}^{-1})\rho_{k}\rho_{k-1}^{2}\cdots\rho_{1}^{2}\tau_{nk}\cdots\tau_{nn}\\
&=\rho_{k}B_{12}^{-1}(B_{23}^{-1}B_{13}^{-1})\cdots(B_{k-2,k-1}^{-1}\cdots B_{1,k-1}^{-1})\rho_{k-1}^{2}\cdots\rho_{1}^{2}\tau_{nk}\cdots\tau_{nn}\\
&=\rho_{k}B_{12}^{-1}(B_{23}^{-1}B_{13}^{-1})\cdots(B_{k-3,k-2}^{-1}\cdots B_{1,k-2}^{-1})\tau_{n,k-1}\rho_{k-2}^{2}\cdots\rho_{1}^{2}\tau_{nk}\cdots\tau_{nn}\\
&=\rho_{k}B_{12}^{-1}(B_{23}^{-1}B_{13}^{-1})\cdots(B_{k-3,k-2}^{-1}\cdots B_{1,k-2}^{-1})\rho_{k-2}^{2}\cdots\rho_{1}^{2}\tau_{n,k-1}\cdots\tau_{nn}\\
&=\rho_{k}\tau_{n}.
\end{align*}
\end{proof}
\begin{Theorem}
For $n\geq2$, the center of $P_{n}(\mathbb{R}P^{2})$ is generated by $\tau_{n}$.
\end{Theorem}
\begin{proof}
We apply an induction on the number $n$ of strands. For $n=2,P_{2}(\mathbb{R}P^{2})$ is the quaternion group with the presentation
$$<\rho_{1},\rho_{2}|\rho_{1}^{2}=\rho_{2}^{2},\rho_{1}^{4}=1,\rho_{1}\rho_{2}\rho_{1}^{-1}=\rho_{2}^{-1}>.$$
Then $Z(P_{2}(\mathbb{R}P^{2}))=<\rho_{1}^{2}=\rho_{2}^{2}|\rho_{1}^{4}=1>$, where $\tau_{21}=B_{12}=\rho_{1}^{2}$ and $\tau_{22}=B_{12}^{-1}\rho_{2}^{2}=1$.
The Fadell-Neuwirth fibration gives us the exact sequence.
$$1\rightarrow\pi_{1}(\mathbb{R}P^{2}-\{\mathbf{x_{1}},\cdots,\mathbf{x_{n}}\})\xrightarrow{\delta_{\ast}} P_{n+1}(\mathbb{R}P^{2})\xrightarrow{\theta_{\ast}} P_{n}(\mathbb{R}P^{2})\rightarrow 1$$
where $\pi_{1}(\mathbb{R}P^{2}-\{\mathbf{x_{1}},\cdots,\mathbf{x_{n}}\})$ is a free group, for $n\geq 2$. Suppose the center of $P_{n}(\mathbb{R}P^{2})$ is generated by $\tau_{n}$. Since $\theta_{\ast}$ is surjective and $\theta_{\ast}(\tau_{n+1})=\tau_{n}$, $Z(P_{n+1}(\mathbb{R}P^{2}))$ is generated by $\tau_{n+1}$ and generators of $\ker \theta_{\ast}=im\delta_{\ast}$, which is a free group. Thus the center of $P_{n+1}(\mathbb{R}P^{2})$ is just generated by $\tau_{n+1}$.
\end{proof}
$\tau_{n}$ can be represented by $h(t)$ for $t\in[0,1]$:
$$h(t)=(h_{1}(t),\cdots,h_{n}(t)).$$
The components of $h(t)$ are plotted in Figure 3.
\begin{figure}[h]
  \centering
  \includegraphics[]{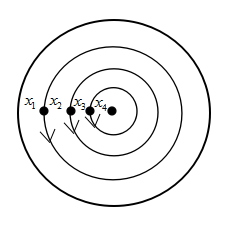}\\
  \caption*{Figure 3}
\end{figure}
\section{The relationship between orbit braid groups and equivariant mapping class groups on compact closed surfaces}
Let $M$ be a compact closed surface and $G$ be a discrete group acting on $M$ freely and properly. Since $M$ is compact, $G$ is finite. Let $\mathbf{x_{1}},\cdots,\mathbf{x_{n}}$ denote n fixed but arbitrarily chosen points on $M$, satisfying $G\mathbf{x_{i}}\bigcap G\mathbf{x_{j}}=\emptyset$,for $1\leq i\neq j \leq n$.
\begin{itemize}
  \item Denote $\mathcal{F}_{n}^{G}M$ as the group of all $G$-homemorphisms $h:M\rightarrow M$ which satisfy $h(G\mathbf{x_{i}})=G\mathbf{x_{i}}$, for each $i$ and preserve orientation if $M$ is oriented.
  \item Denote $\mathcal{B}_{n}^{G}M$ as the group of all $G$-homemorphisms $h:M\rightarrow M$ which satisfy  $h(G\{\mathbf{x_{1}},\cdots,\mathbf{x_{n}}\})=G\{\mathbf{x_{1}},\cdots,\mathbf{x_{n}}\}$ and preserve orientation if $M$ is oriented.
\end{itemize}
These two groups are to be endowed with compact-open topology. We define equivariant mapping class groups as follows:
\begin{itemize}
  \item Denote $\pi_{0}(\mathcal{F}_{n}^{G}M,id)$ as the group of all path connected components of $\mathcal{F}_{n}^{G}M$.
  \item Denote $Mod_{G}(M,n)$ as the group of all path connected components of $\mathcal{B}_{n}^{G}M$ and $Mod_{G}(M,n)$ is called the $n$-th $G$-equivariant mapping class group of $M$.
\end{itemize}
\begin{Remark}
The group structures on $\pi_{0}(\mathcal{F}_{n}^{G}M,id)$ and $Mod_{G}(M,n)$ inherit the group structures on $\mathcal{F}_{n}^{G}M$ and $\mathcal{B}_{n}^{G}M$.
\end{Remark}
Since the action of $G$ on $M$ is free, the (pure) orbit braid group of $M$ is isomorphic to the (pure) braid group of the quotient space~\cite{ref16}, which is an oriented or nonorientable closed surface. Hence we only need to consider the (pure) braid group of the quotient space. Smilar to the situation without group actions, we define the pure evaluation map as follows:
\begin{Definition}
\begin{align*}
\varepsilon^{G}:\mathcal{F}_{0}^{G}M&\rightarrow F(M/G,n)\\
f&\mapsto([f(\mathbf{x_{1}})],\cdots,[f(\mathbf{x_{n}})]).
\end{align*}
\end{Definition}
Observe that we endow $\mathcal{F}_{0}^{G}M$ with compact-open topology and $F(M/G,n)$ with subspace topology of $M\times\cdots\times M$. Then $\varepsilon^{G}$ is continuous.
\begin{Lemma}
The pure evaluation map $\varepsilon^{G}$ has a local cross section.
\end{Lemma}
\begin{proof}
Let $\boldsymbol{a}=([\boldsymbol{a}_{1}],\cdots,[\boldsymbol{a}_{n}])$ be an arbitrary point in $F(M/G,n)$ and $\boldsymbol{a}_{i}$ be the representative of each coordinate, for $i=1,\cdots,n$. Since the action of $G$ on $M$ is free and proper, there exist disjoint open sets $\{U_{i}:\boldsymbol{a}_{i} \in U_{i},i=1,\cdots,n\}$ such that $U_{i}\bigcap gU_{j}=\emptyset$ for every $g\neq e\in G$. We consider an open neighborhood $U({a})$ of the point $\boldsymbol{a}\in F(M/G,n)$ where $U(\boldsymbol{a})$ is defined by
$$U(\boldsymbol{a})=\{([\mathbf{u_{1}}],\cdots,[\mathbf{u_{n}}]):\mathbf{u_{i}}\in U_{i}\}.$$
Let $\mathbf{u}=([\mathbf{u_{1}}],\cdots,[\mathbf{u_{n}}])$ be an arbitrary point in $U({a})$. Choose $\lambda_{\mathbf{u}}$ to be a $G$-equivariant homeomorphism from $M$ to $M$ such that
\begin{enumerate}[(\romannumeral1)]
  \item $\lambda_{\mathbf{u}}(\mathbf{x})=\mathbf{x}$, when $\mathbf{x}\notin\bigcup\limits_{g\in G}\bigcup\limits_{i=1}^{n}gU_{i}$.
  \item $\lambda_{\mathbf{u}}$ maps each $gU_{i}$ into itself homeomorphically;
  \item $\lambda_{\mathbf{u}}$ fixes the points on the boundary of $gU_{i}$;
  \item $\lambda_{\mathbf{u}}(\boldsymbol{a}_{i})$ and $\mathbf{\mathbf{u_{i}}}$ are in the same orbit.
\end{enumerate}
Then the map
$$\chi:U(\boldsymbol{a})\rightarrow \mathcal{F}_{0}^{G}M$$
defined by $\chi(\mathbf{u})=\lambda_{\mathbf{u}}$ is the required cross section.
\end{proof}
\begin{Lemma}
The pure evaluation map $\varepsilon^{G}$ is a locally trivial fiber bundle with fiber $\mathcal{F}_{n}^{G}M$.
\end{Lemma}
\begin{proof}
$\mathcal{F}_{n}^{G}M$ is a closed subgroup of $\mathcal{F}_{0}^{G}M$ and $F(M/G,n)$ is a paracompact space. Thus we only need to show that the spaces $\mathcal{F}_{0}^{G}M/\mathcal{F}_{n}^{G}M$ and $F(M/G,n)$ are homeomorphic~\cite{ref18}. If $h(\mathbf{x})$ and $h'(\mathbf{x})$ belong to the same right coset of $\mathcal{F}_{0}^{G}M$ in $\mathcal{F}_{n}^{G}M$, they must have the property $Gh(\mathbf{x_{i}})=Gh'(\mathbf{x_{i}})$ for $i=1,\cdots,n$. On the one hand, each point $[h]\in \mathcal{F}_{0}^{G}M/\mathcal{F}_{n}^{G}M$ can be associated in a unique manner with a single point $([h(\mathbf{x_{1}})],\cdots,[h(\mathbf{x_{n}})])\in F(M/G,n)$. On the other hand, for each point
$$([\boldsymbol{a}_{1}],\cdots,[\boldsymbol{a}_{n}])\in F(M/G,n)$$
there is a $G$-equivariant homeomorphism $h\in\mathcal{F}_{0}^{G}M$ such that $h(G{a}_{i})=Gx_{i}$ for $i=1,\cdots,n$. Finally, it can be established that the topology of $\mathcal{F}_{0}^{G}M/\mathcal{F}_{n}^{G}M$ coincides with the topology of $F(M/G,n)$.
\end{proof}
Using Lemma 3.2, we obtain an exact sequence as follows:
\begin{equation}
\cdots\rightarrow \pi_{1}(\mathcal{F}_{0}^{G}M)\xrightarrow {\varepsilon_{\ast}^{G}} P_{n}(M/G)\xrightarrow{d_{\ast}^{G}}\pi_{0}(\mathcal{F}_{n}^{G}M)\xrightarrow{i_{\ast}^{G}}\pi_{0}(\mathcal{F}_{0}^{G}M)\rightarrow\pi_{0}(F(M/G,n)=1
\end{equation}
\begin{Lemma}
$\ker d_{\ast}^{G}\subset Z(P_{n}(M/G))$.
\end{Lemma}
\begin{proof}
For any element $\alpha\in \ker d_{\ast}^{G}=im\varepsilon_{\ast}^{G}$, there exists an $H\in \pi_{1}(\mathcal{F}_{0}^{G}M)$ such that $\varepsilon_{\ast}^{G}(H)=\alpha$. Now, $H$ can be represented by some loop $H_{t}:M\rightarrow M$ where each $H_{t}$ is in $\mathcal{F}_{0}^{G}M$ and $H_{0}$ and $H_{1}$ are both identity maps. Then $\alpha$ is represented by $\varepsilon_{\ast}^{G}(H)$ such that
$$\varepsilon_{\ast}^{G}(H)(t)=([H_{t}(\mathbf{x_{1}})],\cdots,[H_{t}(\mathbf{x_{n}})]).$$
Moreover,each $H_{t}$ can induce the unique map from $M/G$ to $M/G$, since it preserves $G$ -action. There is no harm in denoting the induced map by $H_{t}$. Choose any element $\beta\in P_{n}(M/G)$ which can be represented by a loop $(\beta_{1},\cdots,\beta_{n})$ with $\beta_{i}(0)=\beta_{i}(1)=x_{i}$ for each $i=1,\cdots,n$. Use $H_{t}$ and $\beta$ to construct $\psi:I^{2}\rightarrow F_{n}(M/G)$ by
$$\psi(s,t)=([H_{t}(\beta_{1}(s)],\cdots,[H_{t}(\beta_{n}(s)]).$$
Thus, $H(0,t)=H(1,t)$ represents $\alpha$, while $H(s,0)=H(s,1)$ represents $\beta$. So $H_{|\partial I^{2}}$ represents $\alpha\beta\alpha^{-1}\beta^{-1}$. Therefore,
$$\alpha\beta\alpha^{-1}\beta^{-1}=1.$$
Since $\beta\in P_{n}(M/G)$ is arbitrary, it follows
$$\alpha\in Z(P_{n}(M/G)).$$
\end{proof}
Suppose $G$ is a finite group of order $l$, acting freely and properly on $M$. The natural projection map $p:M\rightarrow M/G$ is a regular covering map with the exact sequence
$$1 \rightarrow\pi_{1}(M)\xrightarrow{p_{\ast}}\pi_{1}(M/G)\rightarrow G\rightarrow 1$$
and $G$ is naturally isomorphic to the group of covering transformation. Furthermore, $M/G$  is a closed compact surface whose Eular characteristic $\chi(M/G)$ satisfies the formula
\begin{equation}
l\chi(M/G)=\chi(M).
\end{equation}
If $M$ is the oriented surface $S_{g}$,
$$M/G\cong S_{\frac{g-1}{l}+1} \quad or \quad N_{\frac{2(g-1)}{l}+2}\quad for \quad g\geq1.$$
If $M$ is the nonorientable surface $N_{k}$,
$$M/G\cong S_{\frac{k-2}{2l}+1} \quad or \quad N_{\frac{k-2}{l}+2} \quad for \quad k\geq2.$$
When $M/G\cong S_{g},\quad g\geq2$ or $M/G\cong N_{k},\quad k\geq2$, $M$ can be $S_{g},\quad g\geq 1$ or $N_{k},\quad k\geq 2$ and we have
$$\ker d_{\ast}^{G}\subset Z(P_{n}(M/G))=1.$$
We conclude
\begin{Theorem}
Let $i_{\ast}^{G}:\pi_{0}(\mathcal{F}_{n}^{G}M)\rightarrow\pi_{0}(\mathcal{F}_{0}^{G}M)$ be the homomorphism induced by inclusion $\mathcal{F}_{n}^{G}M\subset \mathcal{F}_{0}^{G}M$.Then
\begin{align*}
\ker i_{\ast}^{G}&\cong P_{n}(M/G),\quad if\quad M/G \quad is \quad S_{g},\quad g\geq2 \quad or \quad N_{k},\quad k\geq2.
\end{align*}
\end{Theorem}
Consider $M$ is $\mathbb{S}^{2}$. Because of the formula (2), the only nontrivial group $G\neq1$ that acts freely on the sphere is $\mathbb{Z}_{2}$. When $G$ is trivial and $n\geq3$, the discussion is the same as that of ~\cite{ref5} and ~\cite{ref6}. And since $Z(P_{1}(\mathbb{S}^{2}))=Z(P_{2}(\mathbb{S}^{2}))=1$, $\ker i_{\ast}^{G}\cong P_{n}(\mathbb{S}^{2})$. We only discuss the situation where the orbit space is $\mathbb{R}P^{2}$.
\begin{Lemma}
If $G=\mathds{Z}_{2}$ acts on $\mathbb{S}^{2}$ antipolarlly, then for $n\geqslant1$
$$Z(P_{n}(\mathbb{R}P^{2}))\subset \ker d_{\ast}^{G}.$$
\end{Lemma}
\begin{proof}
Since $Z(P_{n}(\mathbb{R}P^{2}))$ is generated by $\tau_{n}$ from Theorem 2.4, we only need to prove $d_{\ast}^G(\tau_{n})=1$. Define $H_{t}:\mathbb{S}^{2}\rightarrow \mathbb{S}^{2}$ by
\begin{align*}
H_{t}:(\alpha,\theta)\mapsto(\alpha+2\pi t,\theta),
\end{align*}
where the $i$th component of $\tau_{n}$ can be represented by $[H_{t}(\mathbf{x_{i}})]$, $i=1,\cdots,n$. Clearly $\tau_{n}=\varepsilon_{\ast}^{G}([H_{t}])$. Thus we obtain $\tau_{n}\in\ker d_{\ast}^{G}$.
\end{proof}
As for the nonorientable surface $\mathbb{R}P^{2}$, $G$ must be trivial because of the formula (2).
\begin{Lemma}
If $G$ is trivial and $M$ is $\mathbb{R}P^{2}$, then $Z(P_{n}(\mathbb{R}P^{2}))\subset \ker d_{\ast}^{G}$.
\end{Lemma}
\begin{proof}
Since $Z(P_{n}(\mathbb{R}P^{2}))$ is generated by $\tau_{n}$ from Theorem 2.4, we only need to prove $d_{\ast}^G(\tau_{n})=1$. View $\mathbb{R}P^{2}$ as the quotient space of $\mathbb{S}^{2}$. Define $H_{t}:\mathbb{R}P^{2}\rightarrow \mathbb{R}P^{2}$ by
\begin{align*}
H_{t}:[(\alpha,\theta)]\mapsto[(\alpha+2\pi t,\theta)],
\end{align*}
where the $i$th component of $\tau_{n}$ can be represented by $H_{t}(\mathbf{x_{i}})$, $i=1,\cdots,n$. Clearly $\tau_{n}=\varepsilon_{\ast}^{G}([H_{t}])$. Thus we obtain $\tau_{n}\in\ker d_{\ast}^{G}$.
\end{proof}
We conclude the situation where the quotient space is $\mathbb{S}^{2}$ or $\mathbb{R}P^{2}$:
\begin{Theorem}
Let $i_{\ast}^{G}:\pi_{0}(\mathcal{F}_{n}^{G}M)\rightarrow\pi_{0}(\mathcal{F}_{0}^{G}M)$ be the homomorphism induced by inclusion $\mathcal{F}_{n}^{G}M\subset \mathcal{F}_{0}^{G}M$. Then
\begin{align*}
\ker i_{\ast}^{G}&\cong P_{n}(M/G)/Z(P_{n}(M/G)),\quad if\quad M/G \quad is \quad \mathbb{S}^{2} \quad or \quad \mathbb{R}P^{2}.
\end{align*}
\end{Theorem}
Next ,we consider the situation where $M$ the quotient space $M/G$ is $\mathbb{T}^{2}$. From formula (2), $M$ can be $\mathbb{T}^{2}$ or Klein bottle $N_{1}$ But we claim that
\begin{Lemma}
The Klein bottle $N_{1}$ can not be a cover of the torus $\mathbb{T}^{2}$.
\end{Lemma}
\begin{proof}
$\pi_{1}(N_{1})$ is isomorphic to the non-abelian group $<x,y|x^{2}=y^{2}>$ while $\pi_{1}(\mathbb{T}^{2})$ is isomorphic to the abelian group $\mathds{Z}\bigoplus\mathds{Z}$. Suppose $N_{1}$ is a cover of $\mathbb{T}^{2}$, there exists a injective:
$$\pi_{1}(N_{1})\rightarrow \pi_{1}(\mathbb{T}^{2}).$$
Since the subgroup of an Abelian group is also an Abelian group, the injective map is impossible.
\end{proof}
Thus when $M/G$ is $\mathbb{T}^{2}$, $M$ can only be $\mathbb{T}^{2}$. In order to distinguish the original space and the orbit space, we denote the orbit space by $\overline{\mathbb{T}}^{2}$. Unlike the non-equivariant case, it is possible that $Z(P_{n}(\overline{\mathbb{T}}^{2})\nsubseteq \ker d_{\ast}^{G}$. There exist examples where $Z(P_{n}(\overline{\mathbb{T}}^{2})\nsubseteq \ker d_{\ast}^{G}$.
\begin{Example}
We view the torus as $\mathbb{R}/\mathds{Z}\times \mathbb{R}/\mathds{Z}$ and denote a point in $\mathbb{T}^{2}$ by its rectangular coordinates $([u],[v])$, where $[u],[v]$ are real numbers module 1. Denote each fixed point $\mathbf{x_{i}}$ by $([u_{i}],[v_{i}])$, for $i=1,\cdots,n$. Then $\widetilde{{a}}$ and $\widetilde{b}$ can be represented by $f(t)$ and $g(t)$ for $t\in[0,1]$:
\begin{align*}
f(t)&=([u_{1}-t],[v_{1}]),\cdots,([u_{n}-t],[v_{n}]));\\
g(t)&=([u_{1}],[v_{1}-t]),\cdots,([u_{n}],[v_{n}-t])).
\end{align*}
We can plot the components of $f(t)$ and $g(t)$ in Figure 4, 5 when $n$ is 4.
\begin{figure}[htbp]
\subfigure{
\begin{minipage}[t]{0.5\linewidth}
\centering
\includegraphics[width=1\textwidth]{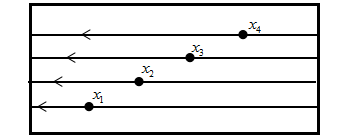}
\caption*{Figure 4}
\end{minipage}}
\subfigure{
\begin{minipage}[t]{0.5\linewidth}
\centering
\includegraphics[width=1\textwidth]{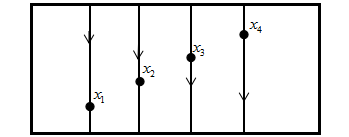}
\caption*{Figure 5}
\end{minipage}}
\end{figure}

Define $\mathds{Z}_{2}$-action on $\mathbb{T}^{2}$ by
\begin{align*}
-1:\mathbb{T}^{2}&\rightarrow \mathbb{T}^{2}\\
([u],[v])&\mapsto([u-\frac{1}{2}],[v]).
\end{align*}
We will compute $d_{\ast}^{\mathds{Z}_{2}}(\widetilde{{a}})$. Denote $\widetilde{{a}}_{i}$ to be the $i$-th component of $\widetilde{{a}}$. Construct the homeomorphsim $H_{t}$ for $t\in[0,1]$ by
\begin{align*}
H_{t}:\mathbb{T}^{2}&\rightarrow \mathbb{T}^{2}\\
([u],[v])&\mapsto([u-\frac{1}{2}t],[v]),
\end{align*}
which satisfies $H_{0}=id$ and $H_{t}(x_{i})=\widetilde{{a}}_{i}$. Since $d_{\ast}^{\mathds{Z}_{2}}(\widetilde{{a}})=H_{1}\neq id$, $Z(P_{n}(\overline{\mathbb{T}}^{2}))\nsubseteq \ker d_{\ast}^{\mathds{Z}_{2}}$.
\end{Example}

Although $Z(P_{n}(\overline{\mathbb{T}}^{2})\nsubseteq \ker d_{\ast}^{\mathds{Z}_{2}}$ in general, we can construct the isomorphism between $\ker d_{\ast}^{G}$ and $imp_{\ast}$.
\begin{Lemma}
$$\ker d_{\ast}^{G}\cong imp_{\ast}.$$
\end{Lemma}
\begin{proof}
Since $\ker d_{\ast}^{G}\subset Z(P_{n}(\overline{\mathbb{T}}^{2}))=<\widetilde{{a}},\widetilde{b}:\widetilde{{a}}\widetilde{b}=\widetilde{b}\widetilde{{a}}>$ and $imp_{\ast}\subset \pi_{1}(\overline{\mathbb{T}}^{2})=<{a},b:{a}b=b{a}>$, we can define $\varphi:\ker d_{\ast}^{G}\rightarrow \pi_{1}(\overline{\mathbb{T}}^{2})$ by
$$\varphi(\widetilde{{a}}^{m}\widetilde{b}^{r})={a}^{m}b^{r}.$$
$\widetilde{{a}}$ and $\widetilde{b}$ can be represented by loops
\begin{align*}
f:&I\rightarrow F(\overline{\mathbb{T}}^{2},n)\\
&t\mapsto([u_{1}-t],[v_{1}]),\cdots,([u_{n}-t],[v_{n}]))
\end{align*}
and
\begin{align*}
g:&I\rightarrow F(\overline{\mathbb{T}}^{2},n)\\
&t\mapsto([u_{1}],[v_{1}-t]),\cdots,([u_{n}],[v_{n}-t]))
\end{align*}
where $f(0)=f(1)=g(0)=g(1)=(\mathbf{x_{1}},\cdots,\mathbf{x_{n}})$. According to the construction of $d_{\ast}^{G}$, there exists the $G$-homeomorphism $H_{t}:\mathbb{T}^{2}\rightarrow \mathbb{T}^{2}$ for $t\in[0,1]$, such that $H_{0}=H_{1}=id$ and $$(p(H_{t}(\mathbf{x_{1}})),\cdots,p(H_{t}(\mathbf{x_{n}})))=f^{m}g^{r}(t)$$
Define the map $H(x_{i}):I\rightarrow \mathbb{T}^{2}$ by $H(\mathbf{x_{i}})(t)=H_{t}(\mathbf{x_{i}})$. Then we obtain that $p_{\ast}([H(\mathbf{x_{i}})])={a}^{m}b^{r}$. Thus ${a}^{m}b^{r}\in imp_{\ast}$.

Conversely, construct $\psi:imp_{\ast}\rightarrow Z(P_{n}(\overline{\mathbb{T}}^{2}))$ by
$$\psi({a}^{m}b^{r})=\widetilde{{a}}^{m}\widetilde{b}^{r}.$$
So there exists a homotopy class $C\in \pi_{1}(\mathbb{T}^{2})$ such that $p_{\ast}(C)={a}^{m}b^{r}$. Now, $C$ can be represented by some loop $c=([c_{1}],[c_{2}]):I\rightarrow \mathbb{T}^{2}$ where $c(0)=c(1)=([0],[0])$. Then construct the $G$-homeomorphism $H_{t}:\mathbb{T}^{2}\rightarrow \mathbb{T}^{2}$ by
$$H_{t}([u],[v])=([u+c_{1}],[v+c_{2}]),$$
which satisfies $H_{0}=H_{1}=id$ and $(p(H_{t}(\mathbf{x_{1}})),\cdots,p(H_{t}(\mathbf{x_{n}})))=f^{m}g^{r}(t)$. Thus $\widetilde{{a}}^{m}\widetilde{b}^{r}\in \ker d_{\ast}^{G}$.
\end{proof}
Since $\mathbb{T}^{2}$ is the covering space of $\overline{\mathbb{T}}^{2}$ and $\pi_{1}(\overline{\mathbb{T}}^{2})$ is an abelian group, each $imp_{\ast}$ is uniquely corresponding to the subgroup of $\pi_{1}(\overline{\mathbb{T}}^{2})\cong\mathds{Z}\bigoplus\mathds{Z}$. And because $G$ is finite, there exist positive integers $q,r$ such that $\widetilde{{a}}^{q},\widetilde{b}^{r}$ generate $imp_{\ast}$. In this case, $G=\mathds{Z}/q\mathds{Z}\bigoplus\mathds{Z}/r\mathds{Z}$ and $imp_{\ast}=<{a}^{q},b^{r}:{a}^{q}b^{r}=b^{r}{a}^{q}>$. We conclude this situation as follows.
\begin{Theorem}
Let $G$ be a discrete group acting on $M$ freely and properly with the quotient space $\overline{\mathbb{T}}^{2}$. Then $M$ must be $\mathbb{T}^{2}$ and each $G$ satisfying the condition is in the form of $\mathds{Z}/q\mathds{Z}\bigoplus\mathds{Z}/r\mathds{Z}$, for positive integers $q,r$. Let $i_{\ast}^{G}:\pi_{0}(\mathcal{F}_{n}^{G}M)\rightarrow\pi_{0}(\mathcal{F}_{0}^{G}M)$ be the homomorphism induced by inclusion. Then
$$\ker i_{\ast}^{G}\cong P_{n}(M/G)/<\widetilde{{a}}^{q},\widetilde{b}^{r}:\widetilde{{a}}^{q}\widetilde{b}^{r}=\widetilde{b}^{r}\widetilde{{a}}^{q}>.$$
\end{Theorem}
Finally we study the relationship between surface orbit braid groups and equivariant mapping class groups. Similarly, we define the evaluation map between $\mathcal{B}_{0}^{G}(M)=\mathcal{F}_{0}^{G}(M)$ and $F(M/G,n)/\Sigma_{n}$.
\begin{Definition}
\begin{align*}
\eta^{G}:\mathcal{B}_{0}^{G}M&\rightarrow F(M/G,n)/\Sigma_{n}\\
f&\mapsto\{[f(\mathbf{x_{1}})],\cdots,[f(\mathbf{x_{n}})]\}
\end{align*}
is called the evaluation map.
 \end{Definition}
Since the evaluation map $\eta^{G}$ is the pure evaluation $\varepsilon^{G}$ associated with the projective map $F(M/G,n)\rightarrow F(M/G,n)/\Sigma_{n}$, we obtain
\begin{Lemma}
$\eta^{G}$ is a locally trivial fiber bundle with fiber $\mathcal{B}_{n}^{G}M$.
\end{Lemma}
Let $j_{\ast}^{G}:\pi_{0}(\mathcal{B}_{n}^{G}M)\rightarrow\pi_{0}(\mathcal{B}_{0}^{G}M)$ be the homomorphism induced by inclusion $\mathcal{B}_{n}^{G}M\subset \mathcal{B}_{0}^{G}M$. Then there is a long exact sequence
\begin{equation}
\cdots\rightarrow\pi_{1}\mathcal{B}_{0}^{G}M\xrightarrow{\eta_{\ast}^{G}} B_{n}(M/G)\xrightarrow{\zeta_{\ast}^{G}}\pi_{0}\mathcal{B}_{n}^{G}M\xrightarrow{j_{\ast}^{G}}\pi_{0}\mathcal{B}_{0}^{G}M\rightarrow\pi_{0}(F(M/G)/\Sigma_{n})=1
\end{equation}
\\By observing the exact sequence(1) and (3), we obtain:
$$\ker \zeta_{\ast}^{G}=im \eta_{\ast}^{G}=im \varepsilon_{\ast}^{G}=\ker d_{\ast}^{G}.$$
From the previous discussion about $\ker d_{\ast}^{G}$ before, we can conclude:
\begin{Theorem}
Let $j_{\ast}^{G}:\pi_{0}(\mathcal{B}_{n}^{G}M)\rightarrow\pi_{0}(\mathcal{B}_{0}^{G}M)$ be the homomorphism induced by inclusion $\mathcal{B}_{n}^{G}M\subset \mathcal{B}_{0}^{G}M$. Then
\begin{align*}
\ker j_{\ast}^{G}&\cong B_{n}(M/G),\quad if\quad M/G \quad is \quad S_{g},\quad g\geq2 \quad or \quad N_{k},\quad k\geq2;\\
\ker j_{\ast}^{G}&\cong B_{n}(M/G)/Z(P_{n}(M/G)),\quad if\quad M/G \quad is \quad \mathbb{S}^{2} \quad or \quad \mathbb{R}P^{2};\\
\end{align*}
When $M/G$ is $\mathbb{T}^{2}$,
\begin{align*}
M&=\mathbb{T}^{2},\quad G=\mathds{Z}/q\mathds{Z}\bigoplus\mathds{Z}/r\mathds{Z},\quad q,r\geq1 \quad and\\
\ker j_{\ast}^{G}&\cong B_{n}(M/G)/<\widetilde{{a}}^{q},\widetilde{b}^{r}:\widetilde{{a}}^{q}\widetilde{b}^{r}=\widetilde{b}^{r}\widetilde{{a}}^{q}>. \end{align*}
\end{Theorem}
%8到5，2-7到7-12，9-15到13-19


\begin{thebibliography}{99}
\bibitem{ref1}E. Artin, Theorie der Z\"{}pfe, \emph{Abh. Math. Sem. Univ. Hamburg}, \textbf{4}(1925), 47-72.
\bibitem{ref2}E. Artin, Theory of braids, \emph{Ann. Math.} \textbf{48}(1947), 101-126.
\bibitem{ref3}E. Artin, Braids and permutations, \emph{Ann.Math.} \textbf{48}(1947), 643-649.
\bibitem{ref4}P. Bellingeri, E. Godelle, On presentations of surface braid groups, \emph{Journal of Knot Theory and Its Ramifications}, \textbf{16}(09)2007, 1219-1233.
\bibitem{ref5}J. S. Birman, Mapping class groups and their relationship to braid groups, \emph{Commnications on Pure and Applied Mathematics}, Vol.\uppercase\expandafter{\romannumeral5}\uppercase\expandafter{\romannumeral5}\uppercase\expandafter{\romannumeral2}(1969), 213-238.
\bibitem{ref6}J. S. Birman, \emph{Braids,Links and Mapping Class Group}, Princeton University Press, New Jersey, 1974.
\bibitem{ref7}J. V. Buskirk, Braid groups of compact 2-manifolds with elements of finite order, \emph{Trans.Amer.Math.Soc.} \textbf{122}(1966), 81-97.
\bibitem{ref8}E. Fadell, Homotopy groups of configuration spaces and the string problem of Dirac, \emph{Duke Math.J.} \textbf{29}(1962), 231-242.
\bibitem{ref9}E. Fadell, L. Neuwirthh, Configuration spaces, \emph{Math.Scand.} \textbf{10}(1962), 111-118.
\bibitem{ref10}E. Fadell, J. V. Buskirk, The braid groups of $\mathbb{E}^{2}$ and $\mathbb{S}^{2}$, \emph{Duke Math.J.} \textbf{29}(1962), 243-257.
\bibitem{ref11}R. H. Fox, L. Neuwirth, The braid groups, \emph{Math.Scand.} \textbf{10}(1962), 119-126.
\bibitem{ref12}D. L. Goncalves, J. Guaschi, The braid groups of the projective plane and the Fadell Neuwirth short exact sequence, \emph{Geom.Dedicata}, \textbf{130}(2007), 93-107.
\bibitem{ref13}D. L. Goncalves, J. Guaschi, Surface braid groups and coverings, \emph{Journal of the London Mathematical Society}, \textbf{85}(3)2012, 855-868.
\bibitem{ref14}D. L. Goncalves, J. Guaschi, M. Maldonado, Embeddings and the (virtual) cohomological dimension of the braid and mapping class groups of surfaces, \emph{Confluentes Mathematici}, \textbf{10}(1)(2016).
\bibitem{ref15}J. M. Lee, \emph{Quotient Manifolds}, Springer, New York, 2013.
\bibitem{ref16}Hao Li ,Zhi L\"u ,Fengling Li , A theory of orbit braids, arXiv: 1903.11501.
\bibitem{ref17}K. Murasugi, The center of a group with a single defining relation, \emph{Mathematische Annalen}, \textbf{155}(3)(1964), 246-251.
\bibitem{ref18}N. Steenrod, \emph{The topology of fibre bundles}, Princeton University Press, New Jersey, 1974.
\bibitem{ref19}O. Zariski, The topological discriminant group of a Riemann surface of genus $p$, \emph{Amer.J.Math.} \textbf{59}(1937), 335-358.
\end{thebibliography}
\end{document}